\documentclass[12pt]{amsart}
\usepackage{geometry}                
\usepackage{graphicx}

\usepackage{amsmath}
\usepackage{amsfonts}
\usepackage{enumerate}
\usepackage{amssymb}
\usepackage{epstopdf}
\usepackage{bm}
\usepackage{hyperref}

\newtheorem{theorem}{Theorem}[section]
\newtheorem{proposition}[theorem]{Proposition}
\newtheorem{lemma}[theorem]{Lemma}
\newtheorem{corollary}[theorem]{Corollary}

\theoremstyle{definition}
\newtheorem{definition}[theorem]{Definition}

\author{Gregory R. Chambers}
\author{Yevgeny Liokumovich}

\begin{document}

\title[Converting homotopies to isotopies
and dividing homotopies in half]{Converting homotopies to isotopies
and dividing homotopies in half in an effective way}

\newcommand{\R}{\mathbb R}

\maketitle

\begin{abstract}
We prove two theorems about homotopies of curves on $2$-dimensional Riemannian manifolds.
We show that, for any $\epsilon > 0$, if two simple closed curves are homotopic 
through curves of bounded length $L$,
then they are also isotopic through curves of length bounded by $L + \epsilon$. 
If the manifold is orientable, then for any $\epsilon > 0$ we show that, if we can contract a curve $\gamma$ traversed twice through
curves of length bounded by $L$, then we can also contract $\gamma$ through 
curves bounded in length by $L + \epsilon$.

Our method involves cutting curves at their self-intersection points and reconnecting them
in a prescribed way. We consider the space of all curves obtained in this way from the original homotopy, and use a novel approach 
to show that this space contains a path which yields the desired homotopy.
\end{abstract}

\section{Introduction}

This article is devoted to proving two theorems about homotopies of curves on 
a $2$-dimensional Riemannian manifold.  
Throughout this article,
if $\gamma$ is a homotopy, then we denote the curve at time $t$ by $\gamma_t$.

\begin{theorem}
\label{thm:baer_quantified}
Let $M$ be a $2$-dimensional Riemannian manifold with or without boundary, and let $\gamma_0$
and $\gamma_1$ be non-contractible simple closed curves which are homotopic through
curves bounded in length by $L$ via a homotopy $\gamma$.
For any $\epsilon > 0$, we then have that there is an isotopy $\overline{\gamma}$ from $\gamma_0$ to $\gamma_1$ through curves of length at most $L + \epsilon$.
Here, a simple curve is one with no self-intersections, and an isotopy is a homotopy through embeddings (simple
curves).
\end{theorem}

Theorem \ref*{thm:baer_quantified} can be viewed as a quantitative version of a classical theorem of 
R. Baer and D. B. A. Epstein. In \cite{Ba}, Baer proved that if two non-contractible 
closed curves on an orientable closed surface are homotopic then they are isotopic.
Epstein extended this result to non-orientable surfaces and surfaces with boundary
in \cite{E}.

Our second main theorem answers a question of R. Rotman.

\begin{theorem}
\label{thm:no_torsion_quantified}
Let $M$ be an orientable $2$-dimensional smooth manifold with or without boundary.
  Fix a closed curve $\alpha$ on $M$, and
define $2 \alpha$  to be the curve formed by traversing $\alpha$ twice.  If $2 \alpha$ can be
contracted through curves of length at most $L$, then for any $\epsilon > 0$, $\alpha$ can also
be contracted through curves
of length at most $L + \epsilon$.
\end{theorem}

This result can be viewed as a quantitative version of
the fact that an orientable surface does not have elements of
order 2 in its fundamental group.

This theorem is related to an open problem due to N. Hingston and H.-B. Rademacher
(see Remark 3.4 in \cite{HR} and Question 3.2.1 in \cite{BM}).
They ask whether there exists a metric on $S^n$ and
a homology class $X$ of the free loop space $\Lambda S^n$
such that the minimax level of $mX$ is strictly smaller
than the minimax level of $X$, where $m$ is an integer.
A parametric version of Theorem \ref*{thm:no_torsion_quantified} would yield
a negative answer to this question in the first non-trivial case
when $X$ is the fundamental class and $m=n=2$.

The proof of Theorem \ref*{thm:baer_quantified} follows from a more general, purely topological theorem.
To state it, we will first need a definition.

\begin{definition}
\label{defn:epsilon_close}
Let $-\gamma$ denote the curve $\gamma$ with reversed orientation. 
Two curves $\alpha$ and $\beta$ are \textit{$\epsilon-$image equivalent}, $\alpha \sim_\epsilon \beta$,
 if there exits a finite collection of disjoint intervals 
 $\bigsqcup_{i=1} ^n I_i \subseteq S^1$, such that
$length(\alpha(S^1 \setminus \bigsqcup I_i)) + length(\beta(S^1 \setminus \bigsqcup I_i)) < \epsilon$.  We also require
that there exists a permutation $\sigma$ of $\{1, \dots, n \}$ and a map $f:\{1 , \dots, n\} \rightarrow \{0,1\}$,
 such that $\alpha|_{I_i}= (-1)^{f(i)} \beta|_{I_{\sigma(i)}}$ for all $i$.
\end{definition}

The topological theorem can now be stated:

\vspace{2mm}

\noindent \textbf{Theorem $\bm{1.1'}$.}\emph{
Suppose $\gamma$ is a smooth homotopy of closed curves on a 2-manifold $M$
and $\gamma_0$ is a simple closed curve.
Then, for every $\epsilon > 0$,
there exists an isotopy $\overline{\gamma}$ such that $\overline{\gamma}_0 = \gamma_0 $ and
$\overline{\gamma}_1$ is $\epsilon-$image equivalent to a small perturbation of $\gamma_1$. 
 Additionally, for every $t$ there exists a $t'$ such that
$\overline{\gamma}_t$ is $\epsilon-$image equivalent to a small perturbation of $\gamma_{t'}$.
If $\gamma_1$ is simple or is a point, then this homotopy also ends at $\gamma_1$, up to a change in orientation.
}
\vspace{2mm}

In fact, we give an explicit algorithm for constructing the isotopy $\overline{\gamma}$
by producing a continuous process that involves cutting
curves of $\gamma$ into arcs and then by gluing these arcs
together in such a way as to remove self-intersections.
Each curve $\overline{\gamma_t}$ is obtained from some curve 
$\gamma_{t'}$ through this cutting and regluing process, yielding the desired length bound for the isotopy.  In the isotopy that we produce,
there can be several curves $\{\overline{\gamma_{t_i}}\}$
obtained from the same curve $\gamma_{t'}$.

Theorem $1.1^\prime$ restricted to when $\gamma$ is a contraction is used in the work
of one of the authors (G. R. Chambers) and R. Rotman \cite{CR} to turn contractions of simple curves
through free loops of bounded length into contractions through based loops of bounded length.

As a side note, this statement also yields a new elementary proof of the Jordan-Schoenflies theorem (see Corollary \ref*{thm:jordan_schoenflies}).

Theorem \ref*{thm:no_torsion_quantified} also follows from a more general topological theorem.  To state it, we again
first need a definition.

\begin{definition}
\label{defn:subcurves}
We say that the homotopy $\beta$ ``goes through subcurves of $\gamma$"
if, for every $\beta_t$, there is a connected subset $S \subset S^1$ and some $t'$
such that $\gamma_{t'}(S)$ is a closed curve that is equal to $\beta_t$ 
up to a small perturbation.
\end{definition}

Theorem \ref*{thm:no_torsion_quantified} now follows from:
\vspace{2mm}

\noindent \textbf{Theorem $\bm{1.2'}$.}\emph{
Let $\alpha$ be a closed curve on an orientable surface $M$ and $\gamma$
be a homotopy of $2 \alpha$ to a point.  
There then exists a homotopy $\overline{\gamma}$ that goes through subcurves
of $\gamma$ and contracts $\alpha$ to a point.
}

\vspace{2mm}

The rest of the article is structured as follows.  
In section 2 we prove Theorem $1.1^\prime$ and give a short proof
of Jordan-Schoenflies theorem.
In section 3 we prove Theorem \ref*{thm:baer_quantified}.  
In section 4 we prove Theorem $1.2^\prime$,
and discuss how it implies Theorem \ref*{thm:no_torsion_quantified}.  
In section 5 we discuss some higher-dimensional analogues to Theorem $1.2^\prime$.  This motivates the following
question, which is a generalization of Proposition \ref*{prop:3}:

\vspace{2mm}
\noindent \textbf{Question.} Does there exist a function $f(L)$ such that, for any Riemannian metric
on $S^3$, if $2 \alpha$ can be contracted to a point through curves of length $\leq L$,
then $\alpha$ can also be contracted through curves of length $\leq f(L)$? 

\vspace{0.2in}

\textbf{Acknowledgments}
We would like to thank
Alexander Nabutovsky and Regina Rotman for introducing us to the problems studied in this paper
and for many valuable discussions. We would like to thank Jonguk Yang for sharing Example 1 with us 
and Parker Glynn-Adey for telling us about the work of Baer and Epstein.

The authors were both supported by Natural Sciences and Engineering 
Research Council (NSERC) CGS graduate scholarships, and by Ontario Graduate Scholarships.

\section{Proof of Theorem $1.1^\prime$}

In this section, we will prove Theorem $1.1^\prime$.
Fix smooth, closed curves $\gamma_0$ and $\gamma_1$, with $\gamma_0$ simple.  Let $\gamma$ be a homotopy
from $\gamma_0$ to $\gamma_1$ that passes through curves of length at most $L$, and fix $\epsilon > 0$.  We first want to simplify this
homotopy so as to be more manageable - we want to show that we can perturb $\gamma$ to $\widetilde{\gamma}$ so that all curves still are smooth,
and so that the lengths of curves do not grow by more than $\epsilon$.  Additionally, all but finitely many $\widetilde{\gamma_t}$ will be immersed
curves with transverse self-intersections of multiplicity at most $2$, and each of the finitely many singular events that occur do not do so
concurrently and take one of the
the three forms depicted in Figure \ref*{fig:reidemeister_moves}.  We use the term \emph{Reidemeister moves} to refer to these interactions.  This
term is used because of the obvious similarities between them and the moves used in Reidemeister's Theorem in knot theory.

\begin{figure}[center] 
\includegraphics[scale=0.4]{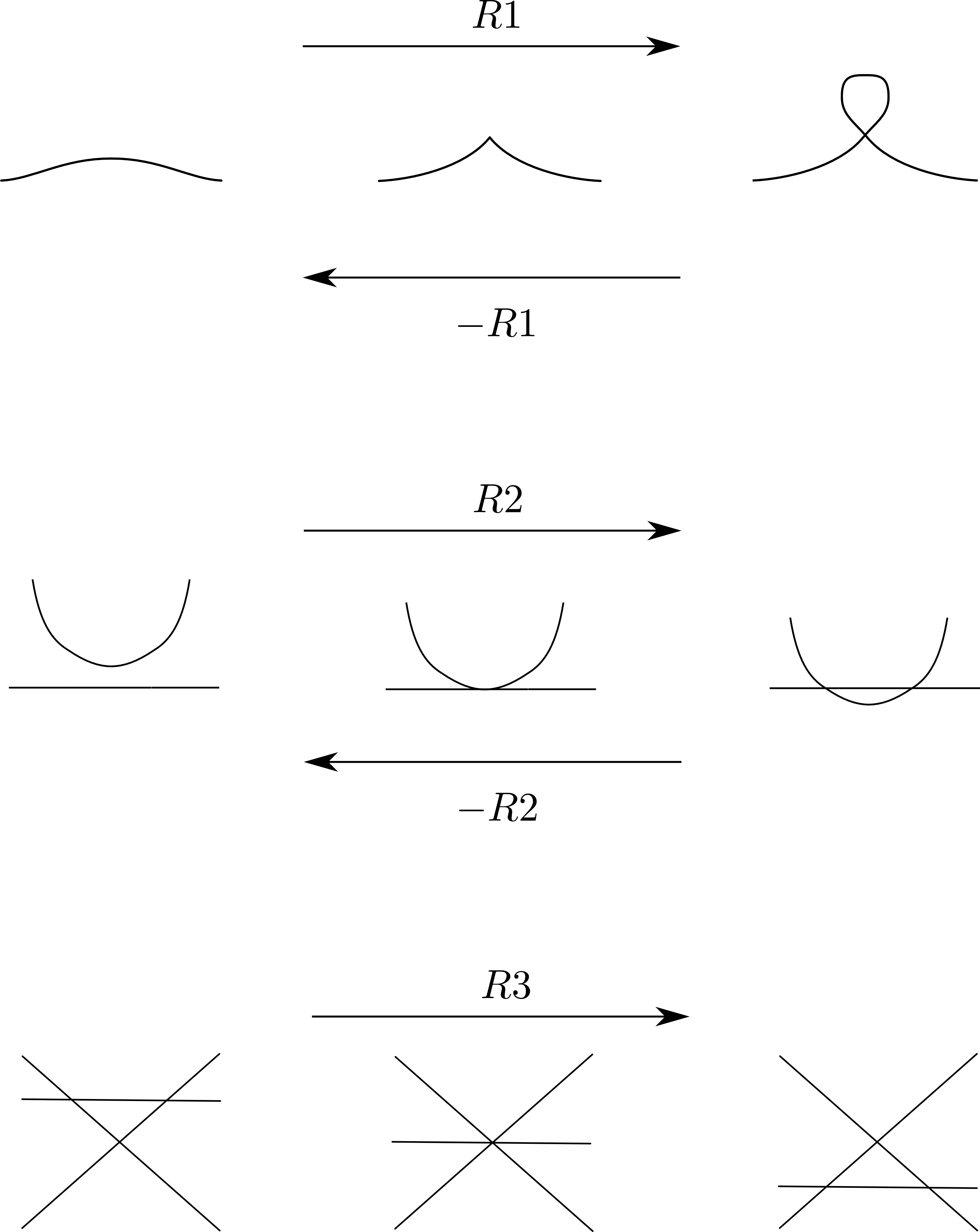} 
\caption{The 3 types of Reidemeister moves} \label{F-R}
\label{fig:reidemeister_moves}
\end{figure}

To prove this perturbation statement rigorously, we could approximate $\gamma$ by a piecewise linear homotopy
as it is often done in the proof of Reidemeister's Theorem in knot theory (cf. \cite{K}).
However, we will instead use standard results in transversality theory.  A good reference for this is
\cite{GG}.  We will say that a collection of homotopies of smooth closed curves in $M$ is \emph{generic} if
it is an open and dense subset of $C^{\infty}([0,1] \times S^1, M)$ with respect to the $C^\infty$ Whitney topology.
We will also say that a generic homotopy has a certain property if the set of homotopies with that property forms a generic set.

\begin{proposition}
\label{prop:generic}
A generic homotopy $\gamma$ has the following properties:
\begin{enumerate}
	\item	For all but finitely many times $t$, $\gamma_t$ is an immersion with only transverse self-intersections.  At each of these intersections,
		only 2 arcs meet.
	\item	The only singular events that can happen are the 3 Reidemeister moves shown in Figure \ref*{fig:reidemeister_moves}.
\end{enumerate}
\end{proposition}

\begin{proof}
We start by proving that, for a generic homotopy of closed curves, 
there will be only finitely many curves with triple self-intersections.

Let $(S^1)^{(3)}$ be a collection of distinct triplets 
	$$ (S^1)^{(3)}=\{(s_1,s_2,s_3) \in T^3 | s_i \neq s_j \text{ for }i \neq j \}.$$
By definition (see \cite{GG}), the 3-fold 0-jet bundle $J_3 ^0 (S^1,M)$ is the 
$9-$dimensional manifold $(S^1)^{(3)} \times M^3$.
For each $t$, the map $\gamma_t$ induces a map 
$j_3 ^0 \gamma_t: (S^1)^{(3)}\rightarrow J_3 ^0 (S^1,M)$
defined by 
	$$ j_3 ^0 \gamma_t(s_1,s_2,s_3) = (s_1,s_2,s_3,\gamma_t(s_1),\gamma_t(s_2),\gamma_t(s_3)). $$

Consider the set $S=\{(s_1,s_2,s_3,y,y,y)\} \subset J_3 ^0 (S^1,M)$.
$S$ is a $5-$dimensional closed submanifold of $J_3 ^0 (S^1,M)$
and $\gamma_t$ has a triple self-intersection
$\gamma_t(s_1)=\gamma_t(s_2)=\gamma_t(s_3)$ if and only if $j_3 ^0 \gamma_t(s_1,s_2,s_3) \in S$.

Define $f: [0,1] \times (S^1)^{(3)}\rightarrow J_3 ^0 (S^1,M)$ by setting
$f(t,x)=j_3 ^0 \gamma_t(x)$. By a parametric version of Thom's Multijet Transversality Theorem
(proved in \cite{Bru}),
for a generic homotopy $\gamma$ the map $f$ will intersect
submanifold $S$ transversely. By dimensional considerations, along with the fact
that $S$ is closed, the intersection $f \pitchfork T$ will be
a finite collection of points. Hence, there will be only finitely many
triple self-intersections.
These triple self-intersections are depicted as an $R3$ move in Figure \ref*{fig:reidemeister_moves}. 

Similarly, by considering $J_4 ^0 (S^1,M)$ we obtain that curves in a
generic homotopy do not have self-intersections of multiplicity higher than $3$.

Next, we prove that there are only finitely many non-transverse self-intersections.
Consider the 2-fold 1-jet bundle $J_2 ^1 (S^1,M)$. Let $C$ be a submanifold
of $J_2 ^1 (S^1,M)$ consisting of points with equal $M$ components
and co-linear tangent vector components. We compute that this space has co-dimension $3$,
so in a generic homotopy there are going to be only finitely many
non-transverse double points.
Non-transverse double points correspond to moves of Type II depicted on 
Figure \ref*{fig:reidemeister_moves}. 

Finally, by considering the jet bundle $J ^2 (S^1,M)$,
we come to the conclusion that generically there will be finitely many 
points $(t,s)$ with $\nabla_s \gamma_t(s)=0$.
We would like to understand what a generic 1-parametric family looks 
like in the neighbourhood of such an isolated singular point.
Problems of this type are studied in bifurcation theory.
For a generic homotopy, it is possible to 
choose local coordinates around the singular point so that $\gamma$
in these coordinates can be expressed in a particular simple form called the normal form.
In \cite{A} and \cite{D}, normal forms for 1-parametric families of plane curves
are calculated.  It is also demonstrated that the only singular event that can happen
are R1 creations or destructions of a self-intersection, as in Figure \ref*{fig:reidemeister_moves}.

\end{proof}

We can now prove the desired approximation lemma:
\begin{lemma}
\label{lem:perturb}
Given a homotopy $\gamma$ from $\gamma_0$ to $\gamma_1$ and an $\epsilon > 0$, we can perturb $\gamma$ to get a homotopy
$\widetilde{\gamma}$ with the following properties:
\begin{enumerate}
	\item	$\widetilde{\gamma_t}$ and $\gamma_t$ differ in length by at most $\epsilon$, for every $t$.
	\item	There are only finitely many values of $t$ at which $\widetilde{\gamma_t}$ is not normal, that is,
		contains either non-transverse self-intersections, or intersections of multiplicity greater than 2.
	\item	Non-normal intersections do not happen concurrently, and each one looks like one of the Reidemeister
		moves shown in Figure \ref*{fig:reidemeister_moves}.
\end{enumerate}
If $\gamma_1$ contains only normal self-intersections, then $\widetilde{\gamma_1} = \gamma_1$.
\end{lemma}
\begin{proof}
This lemma is mostly an application of Proposition \ref*{prop:generic}.  By this proposition, we can find a perturbation which
satisfies all of the above criteria, except the facts that non-normal interactions do not occur concurrently, and that, if
$\gamma_1$ is normal, then $\widetilde{\gamma_1} = \gamma_1$.  Since singular interactions occur only finitely many times,
we can perform small local perturbations to ensure that they do not occur concurrently.
If $\gamma_1$ has only normal self-intersections, then we see that the above perturbations can be chosen so that $\gamma_1$
is not perturbed at all, completing the proof of this lemma.
\end{proof}

With a slight abuse of notation, let us assume that the homotopy $\gamma$ is already perturbed; that
the length of $\gamma_t$ is less than or equal to $L$ for all $t$, and there exists a finite sequence of times
	$$0 = t_0 < \dots < t_n = 1$$
such that, between $t_i$ and $t_{i+1}$, exactly one of the interactions depicted in Figure \ref*{fig:reidemeister_moves} occurs, and
$\gamma_{t_i}$ possesses only normal self-intersections.
We wish to find a homotopy $\overline{\gamma}$ such that $\overline{\gamma}_0 = \gamma_0$, $\overline{\gamma}$ is composed
of simple, closed curves, and for each $t$, there is a $t'$ such that $\overline{\gamma_t}$ is $\epsilon$-image equivalent to $\gamma_{t'}$.

Our strategy to generate this new homotopy  
will be to substitute each of the Reidemeister moves in Figure \ref*{fig:reidemeister_moves} with a move that does not
introduce self-intersections, as in Figure \ref*{fig:F-M}. Observe that there are 3 possibilities for move $M3b$.

\begin{figure}[center]
\includegraphics[scale=0.7]{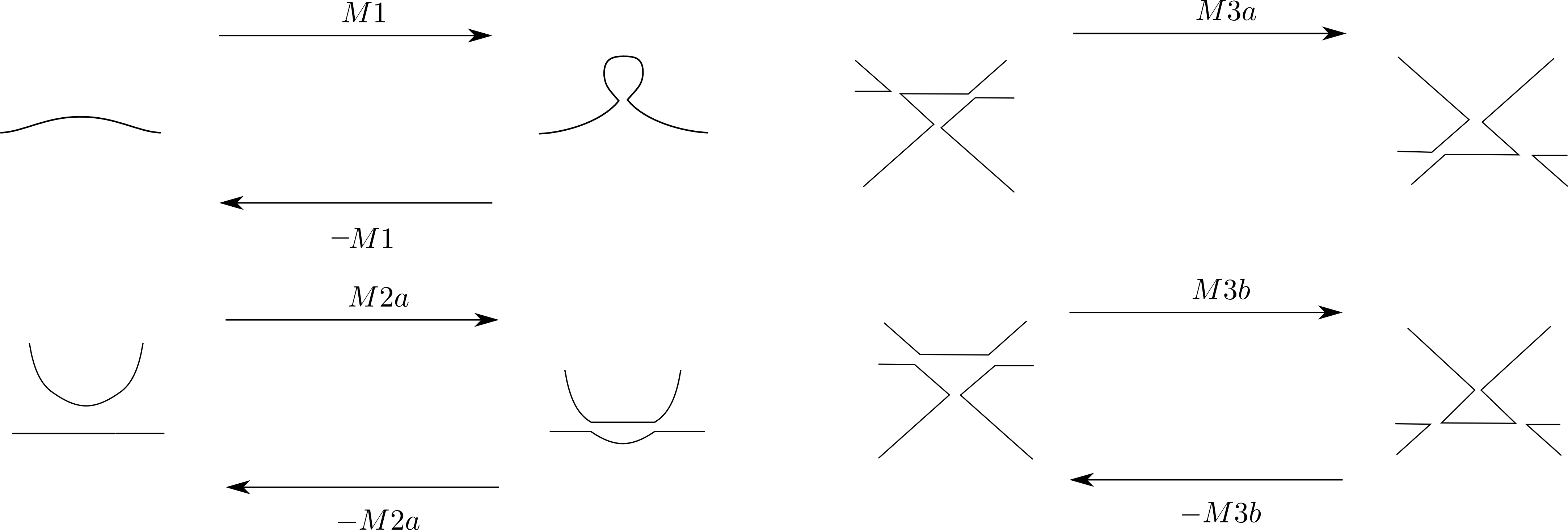} 
\caption{Reidemeister moves modified so as to avoid creating intersections} \label{fig:F-M}
\end{figure}

There is a problem with this naive approach: one of Reidemeister moves cannot be modified
in this way so as to allow the homotopy to move forward and still be intersection-free.
We illustrate the problem and solution with two examples below.

\vspace{0.1in}

\textbf{Example 1.} Jonguk Yang created an example shown in Figure \ref*{fig:jonguk}.
One can substitute the first two Reidemeister moves with moves from Figure \ref*{fig:F-M}, but not the second to last move.
More generally, it is not clear how to proceed when the move $-R2$ occurs in the original homotopy 
and the modified homotopy looks locally like Figure \ref*{fig:problematic}.

\begin{figure}[center]
\includegraphics[scale=0.5]{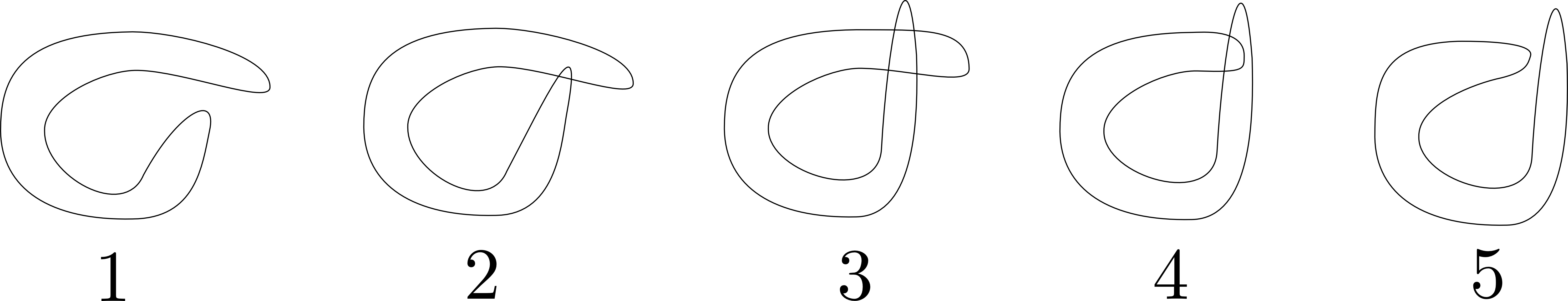} 
\caption{Jonguk Yang's example} \label{fig:jonguk}
\end{figure}

\begin{figure}[center]
\includegraphics[scale=0.5]{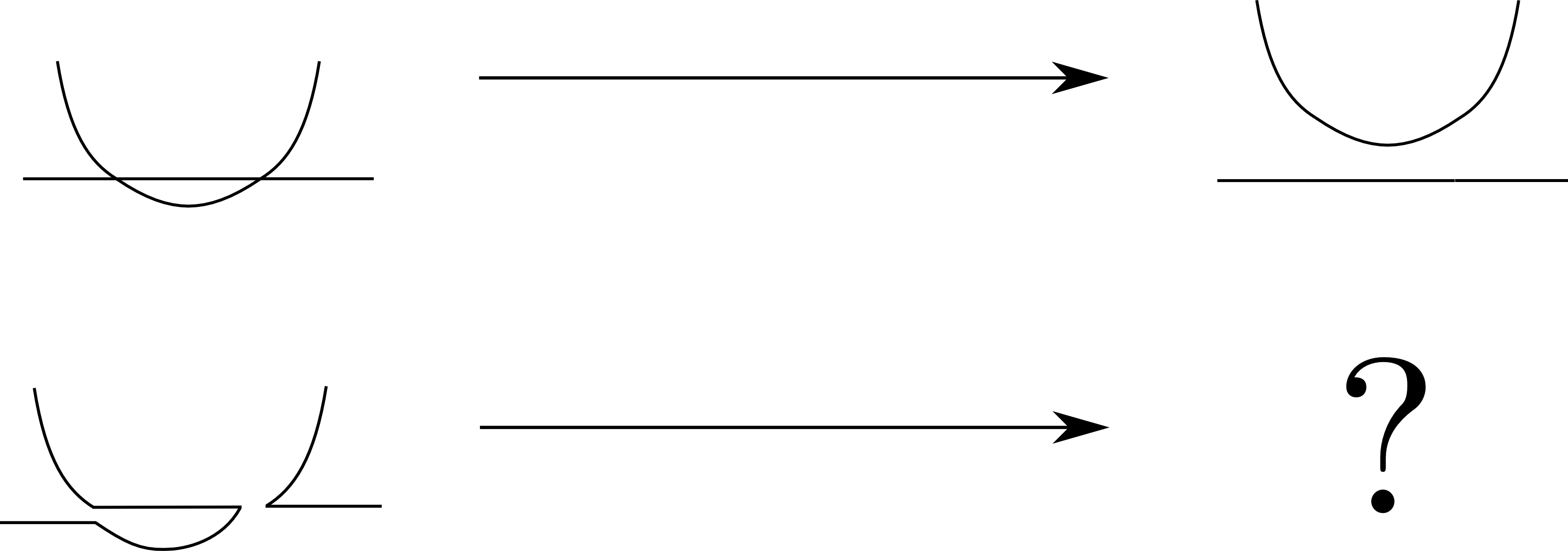} 
\caption{Problematic move} \label{fig:problematic}
\end{figure}

Our solution for Jonguk Yang's example is as follows. Let $t_1$, $t_2$, $t_3$ and $t_4$
be the moments in time when tangential touchings occur in the original homotopy (in chronological order).
We resolve the first two singularities using $M2a$ moves as usual, as in Figure \ref*{fig:F-M}. We then evolve our modified curve until
$t_3$, and then change direction in time, homotoping it back to the contour of the curve at $t_2$,
but changing the way it is connected.  This is shown in Figure \ref*{fig:solution}. We continue to evolve this curve backwards in time
until we get to $t_1$, at which time we apply this technique again, reversing the direction in time and homotoping it forward.
When we reach time $t_3$ again, there is no problem with applying 
move $-M2a$ to the modified curve. At $t_4$, there is also no problem with using a regular $-M2a$ move.
In the end, we resolved the problem by applying the move depicted in Figure \ref*{fig:M2b} twice,
a move that we will refer to as $M2b$.

\begin{figure}[center] 
\includegraphics[scale=0.5]{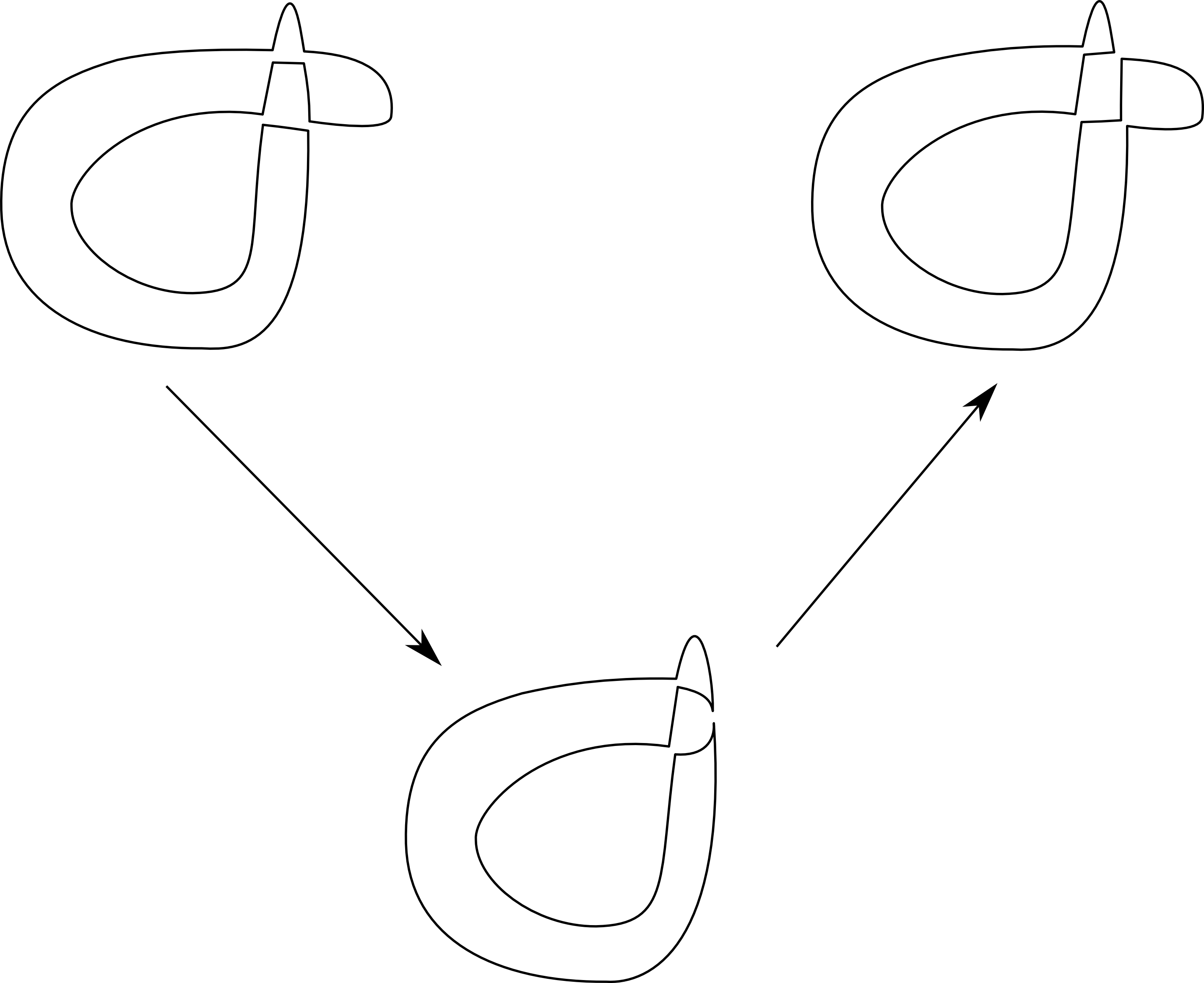} 
\caption{Our solution for Jonguk Yang's example} \label{fig:solution}
\end{figure}

\begin{figure}[center]
\includegraphics[scale=0.7]{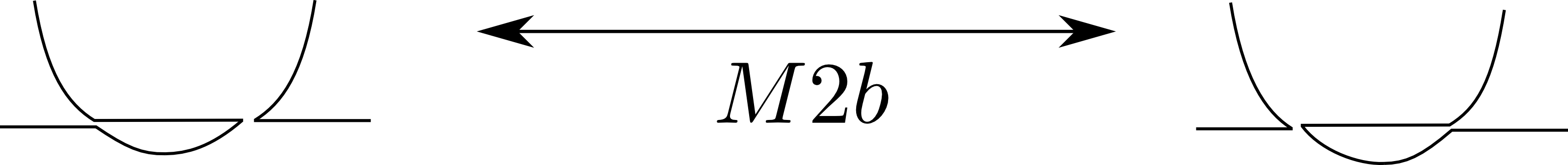} 
\caption{Move $M2b$} \label{fig:M2b}
\end{figure}

\vspace{0.1in}
\textbf{Example 2.} Blindly applying $M2b$ move whenever we encounter a problem
does not always work, as the example in Figure \ref*{fig:example2} illustrates.  

\begin{figure}[center] 
\includegraphics[scale=0.7]{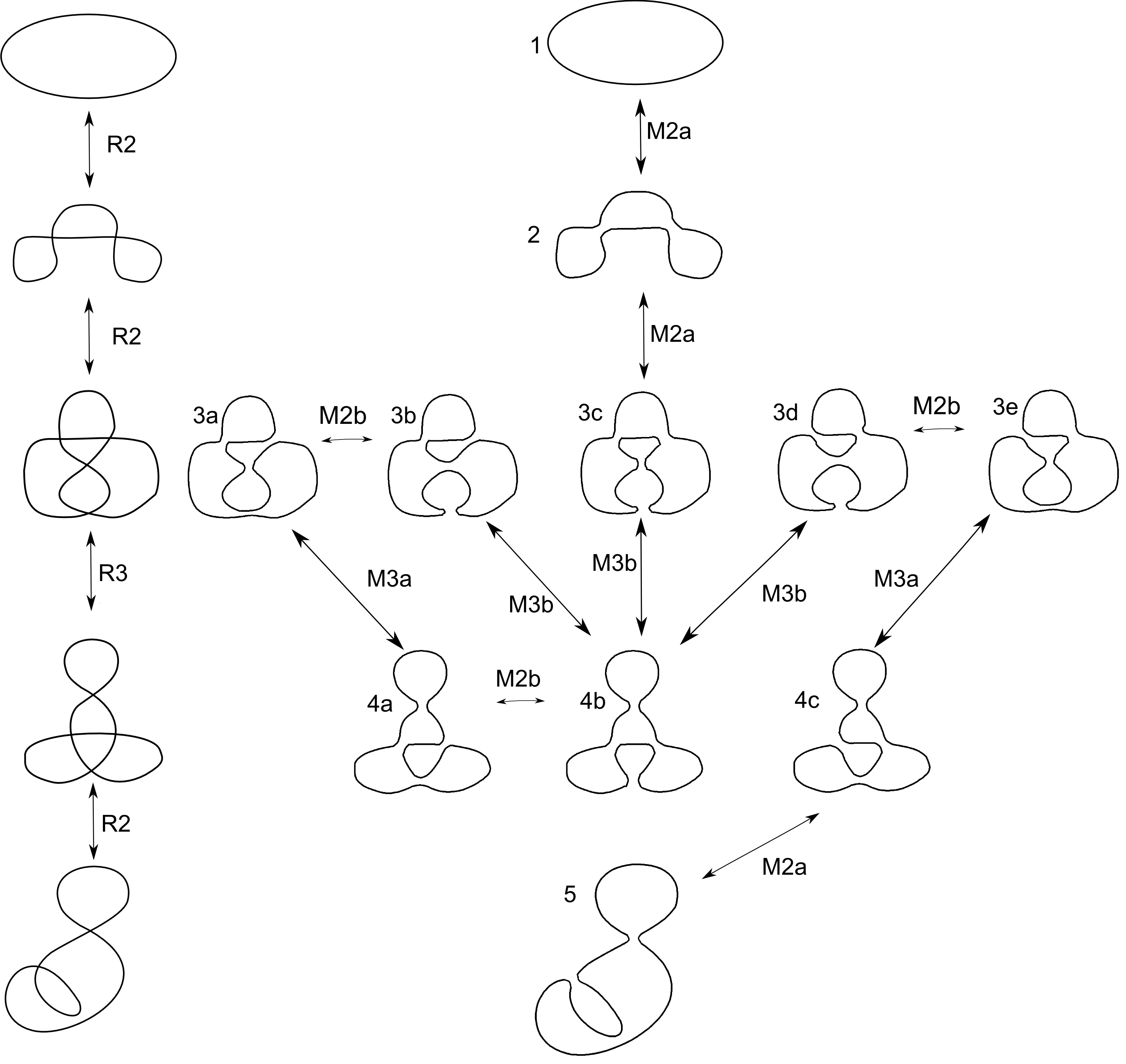} 
\caption{Example 2} \label{fig:example2}
\end{figure}

In the left column on Figure \ref*{fig:example2} there are
five curves from a homotopy that starts on a simple closed curve and
then goes through Reidemeister moves $R2$, $R2$, $R3$, $-R2$. 
On the right we consider all possible ways to redraw these curves
in such a way so that they are connected and do not have self-intersections.
There are 5 different connected redrawings for the third curve
and 3 different redrawings for the fourth curve.

We would like to construct an isotopy from the initial simple curve
to the redrawing of the last curve. By applying moves $M2a$, $M2a$ and $-M3b$
we obtain an isotopy to the curve $4b$, but then we encounter a problematic $-R2$ move as 
in Figure \ref*{fig:problematic}. By applying move $M2b$ we can homotope $4b$ to $4a$,
 but then after moves $M3a$, $M2b$, $-M3a$
we get back to $4b$.
The solution is to use one of $-M3b$ moves to homotope $4b$ to $3d$, followed by a 
sequence of moves that homotope it to the curve $5$.

Observe that all curves, except for $1$ and $5$,
have an even number of arrows pointing to them.
This is a key observation that will be used in the proof below.
\vspace{0.1in}

We now have to show that this technique is versatile enough to handle any issue that might arise.

We now define a procedure of ``resolving" the self-intersections of a curve
$\gamma_{t_j}$.  Let $p$ be a self-intersection point of $\gamma_{t_j}$.
Consider a small ball $B$ around $p$. If we erase all points in $B$ 
from $\gamma_{t_j}$, we obtain two disconnected curves 
$\gamma_{t_j} \setminus (B \cap \gamma_{t_j})$, each with two endpoints.
We can reconnect the endpoints in 
one of the two ways shown in Figure \ref*{fig:resolution}, removing the self-intersection $p$.
If we execute this procedure for every self-intersection of $\gamma_{t_j}$, we either obtain
a single simple closed curve which is $\epsilon$-image equivalent to $\gamma_{t_j}$, or several disconnected simple
closed curves.  We only consider sets of reconnections which result in single simple closed curves, and will call
such a set a \emph{resolution}.

We define a sign of
a reconnection in a resolution as follows. Fix an orientation of the curve. Consider a small neighbourhood 
of the intersection and let $a_1$ and $a_2$ denote two intersecting arcs in this neighbourhood.
If we remove a small ball around the intersection point, each $a_i$ is cut into two subarcs $a_i ^1$ and $a_i ^2$, 
where the order of subarcs is defined by the orientation of $\gamma_{t_j}$. If the resolution
connects $a_1 ^1$ to $a_2 ^2$ we call it positive, and if it connects $a_1^1$ to $a_2^1$ we all it negative.
Observe that this definition of sign is independent of the choice of orientation of the curve.

\begin{figure}[center]
\includegraphics[scale=0.5]{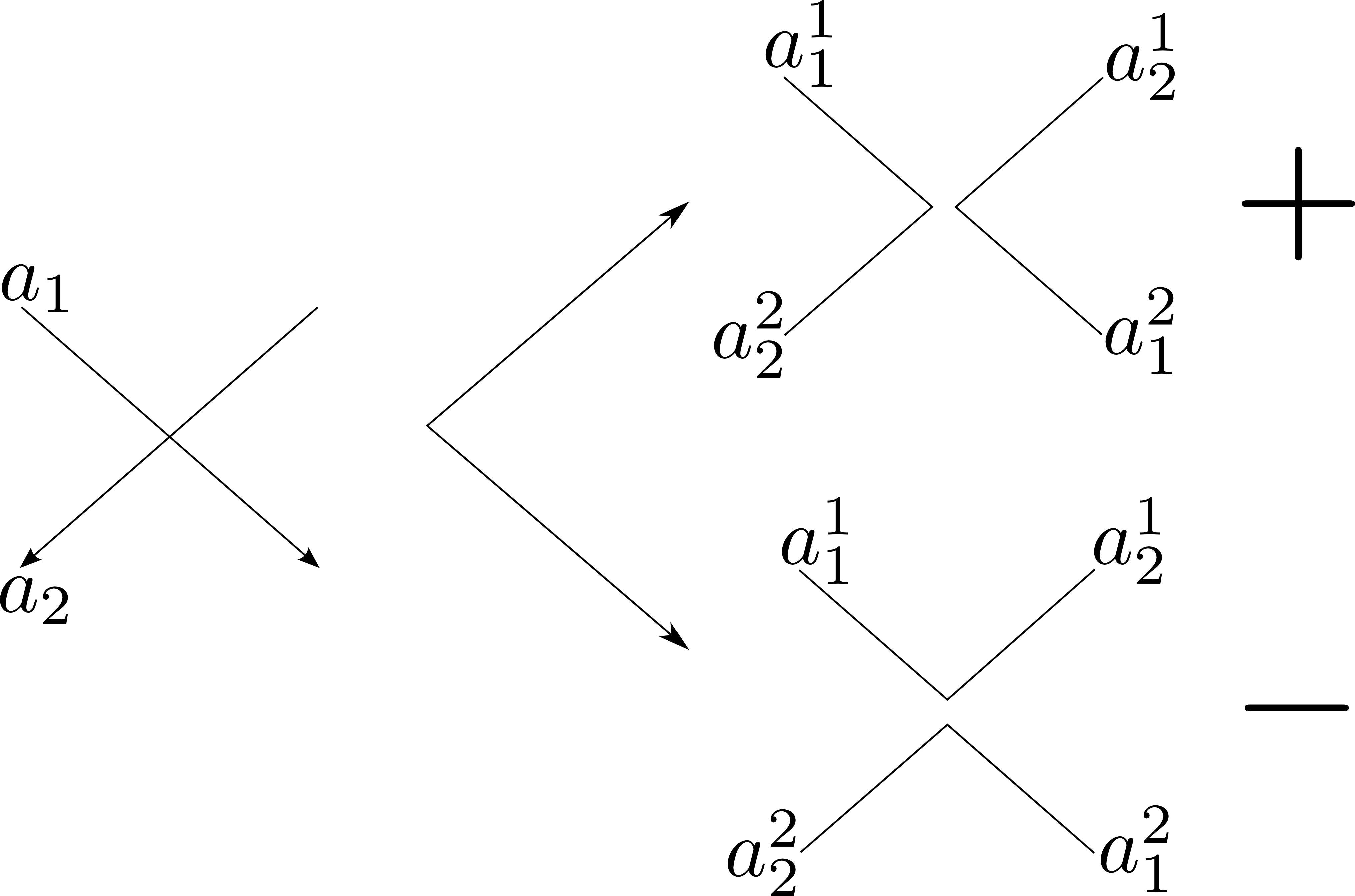} 
\caption{Two ways to reconnect a curve} \label{fig:resolution}
\end{figure}

The isotopy we construct will consist of curves that are obtained from the curves
of the homotopy by resolving all of their self-intersections in the way just described.  
By selecting the balls around
the intersections to be small enough, we ensure that the resolved curves 
are $\epsilon$-image equivalent to the original
curves.

Let $\{G^j _i \}$ denote the set of all connected curves that can be obtained from $\gamma_{t_j}$
by resolving all of its self-intersections. 
Construction of the isotopy relies on connectedness properties of a certain finite graph $\Gamma$.
For each $G^j_i$, we add a vertex to $\Gamma$, and label it ``$G^j_i$".
We will now define edges of $\Gamma$.
It will follow from how these edges are added to $\Gamma$ that, if vertices with labels $G^j _i$
and $G^l_k$ are connected by an edge, then there exists an isotopy 
between corresponding resolutions of $\gamma_{t_j}$ and
$\gamma_{t_l}$ through curves that are $\epsilon-$image equivalent to some
subset of curves in the original homotopy, hence achieving the desired length bound.
We will show that there exists a path in this graph from the vertex corresponding to curve $\gamma_0$ to
a vertex corresponding to a simple closed curve 
which is $\epsilon$-image equivalent to $\gamma_1$.  Since this path
corresponds to an isotopy, we will be done.

All edges that are added to $\gamma$ are either between $G^j_i$ and $G^{j+1}_k$, for $j < n$,
or between $G^j_i$ and $G^j_k$ for $i \neq k$.  We call the first type of edges vertical edges,
and the second type of edges horizontal edges.

\noindent \textbf{Vertical edges}

Let $R_j$ denote the Reidemeister move between times $t_j$ and $t_{j+1}$, $j < n$.
If $R_j= \pm R1$, we connect vertices $G^j _i$ and $G^{j+1}_k$ by an edge 
if the corresponding curves are related by the move $\pm M1$.

If $R_j= R3$, we connect $G^j _i$ and $G^{j+1}_k$ if they are related by $M3a$.
If $G^j_i$ is related to $G^{j+1}_{k_1}$ by $M3b$, then there  must be two other distinct curves
$G^{j+1}_{k_2}$ and $G^{j+1}_{k_3}$ related to $G^j_i$ by $M3b$.
This defines 3 edges in $\Gamma$.
Similarly, if $G^{j+1}_i$ is related to $G^j_{k_1}$ by $-M3b$, then there must be two other curves $G^j_{k_2}$ and $G^j_{k_3}$ related to $G^{j}_i$ by $-M3b$.
This also defines $3$ edges.

If $R_j= \pm R2$ we connect two vertices if the corresponding curves are related by a
$\pm M2$ move. If $G^j_i$ has intersections resolved as in Figure \ref*{fig:problematic} 
and $R_j$ removes intersections in a $-R2$ move, then we do not connect $G^j_i$ to any vertex on the level corresponding to $j + 1$. Similarly, for a curve 
$G^{j+1} _i$ that has intersections resolved as in Figure \ref*{fig:problematic} and created by an $R2$ move, 
we do not connect $G^{j+1}_i$ to any vertex on the $j$th level. Instead, we will connect these vertices with
 a vertex on the same level by a horizontal edge.

\noindent \textbf{Horizontal edges}

Consider a curve $G^{j} _i$ with intersections $a$ and $b$
resolved as in Figure \ref*{fig:problematic}.  We then have that there exists a curve $G^{j}_k$,
which is identical to $G^{j}_i$ except that 
the corresponding intersections have both been resolved with opposite signs. If the preceding move
was an $R2$ creation of $a$ and $b$ we connect $G^j_k$ and $G^j_i$ by an edge. Similarly, if the next move
is a $-R2$ destruction of $a$ and $b$ we also connect $G^j_k$ and $G^j_i$ by an edge. Observe that both of these
may happen. In that case, vertices $G^{j} _i$ and $G^{j} _k$ are connected by 2 edges.

We can now complete the proof of Theorem $1.1^\prime$.
Observe that, by construction, if $j \not \in \{ 0, n \}$, then $G^{j}_i$ is the endpoint of an even number of 
edges.  More specifically, it will be the endpoint of 2, 4, or 6 edges, depending on
the preceding and following Reidemeister moves. Let $\Gamma_0$ be the connected component of the 
graph $\Gamma$ that contains $G^{0}_0 = \gamma_0$. By Euler's degree sum formula (also known as handshaking lemma),
there will be an even number of vertices with odd degree in $\Gamma_0$. Hence, there exists a path $p$
in $\Gamma$ from $G^{0} _0$ to some $G^{n}_i$.

If two vertices of $\Gamma$ are connected by an edge, then clearly there exists a isotopy 
between the corresponding curves described in Figures \ref*{fig:F-M} and \ref*{fig:M2b}.
Hence, $p$ defines a homotopy of simple curves from $\gamma_0$ to a curve $\epsilon-$image equivalent to 
$\gamma_1$ through simple curves $\epsilon-$image equivalent to curves in the homotopy $\gamma$.

The only item left to prove is that, if $\gamma_1$ is simple or is a point, then the isotopy ends at $\gamma_1$.  If it is simple,
then we can perturb our homotopy so that, after it is perturbed, $\gamma_1$ is still the endpoint as per Lemma \ref*{lem:perturb}.  The above procedure does not affect
$\gamma_1$ (up to orientation) and so the isotopy ends at $\pm \gamma_1$, where $-\gamma_1$ is $\gamma_1$ oriented in the opposite direction.
If $\gamma_1$ is a point $p$, then we can perturb our homotopy so slightly so that the endpoint is a simple closed curve which can be contracted
through simple closed curves of length at most $\epsilon$ to $p$.  By the above argument, our isotopy then ends at this small curve, which can then be contracted to
$p$ through small simple closed curves, as desired.  This completes the proof of Theorem $1.1^\prime$.

\vspace{0.1in}

As a corollary we give a short proof of the Jordan-Schoenflies theorem. 

\begin{corollary} \label{thm:jordan_schoenflies}
A smooth simple closed curve in $\R ^2$
divides the plane into two regions diffeomorphic to the outside and inside of the unit
circle.

\end{corollary}

\begin{proof}
Let $\gamma_0$ be a smooth, closed, simple curve in the plane.  
We define a homotopy through smooth curves from the unit circle $S^1$ to $\gamma_0$,
and then apply Theorem $1.1^\prime$ to it.
The homotopy is extremely simple.  We have that
	$$ \gamma_0:	S^1 \rightarrow \mathbb{R}^2, $$
and so we just linearly interpolate between the identity map on $S^1$ and this map.  That is,
for each $t \in [0, 1]$ and each $s \in S^1$, we define
	$$ \gamma(t,s) = s + t ( \gamma_0(s) - s). $$
By Theorem $1.1^\prime$, we can thus find an isotopy $\overline{\gamma}$ from $S^1$
to $\gamma_0$.  Since this isotopy is smooth, using
the Isotopy Extension Theorem (see Theorem 1.3 in \cite{H} on p.180),
we get a diffeomorphism from $\mathbb{R}^2 \setminus S^1$
to $\mathbb{R}^2 \setminus \gamma_0$.  
\end{proof}

\section{A quantitative version of the Baer-Epstein Theorem}

Recall that a theorem of Baer and Epstein states that if two closed curves on a 2-surface $M$
are homotopic, but not contractible, then they are isotopic.
If the two curves are contractible, then the theorem does not hold 
when $M$ is any orientable surface except for a sphere. 
Indeed, for these orientable surfaces, an isotopy preserves orientation, so two null-homotopic curves of opposite 
orientation are homotopic, but not isotopic.

In this section we use Theorem $1.1^\prime$ to prove 
Theorem \ref*{thm:baer_quantified}, an effective version
of the result of Baer and Epstein.

\begin{proof}
If $M$ is an orientable surface then the result follows immediately from
Theorem $1.1^\prime$. Indeed, by this theorem there exists
an isotopy between $\gamma_0$ and $\pm \gamma_1$ through curves of length less than or equal to
$L + \epsilon$.  However, $\gamma_1$ and $-\gamma_1$
are not homotopic, so the isotopy must be between $\gamma_0$ and $ \gamma_1$.

More generally, whenever $\gamma_1$ and $-\gamma_1$ are homotopic, 
the result still follows from Theorem $1.1^\prime$.
Suppose $\gamma_1$ and $-\gamma_1$ are homotopic as free loops, and are non-contractible. Fix a point
$p \in \gamma_1$ and consider the group $\pi_1(M,p)$ of homotopy classes of loops based at
$p$. There then exists a loop $\beta$ based at $p$ such that 
$[\beta \gamma_2 \beta^{-1}]=[-\gamma_2]$. 
By Lemma 2.3 in \cite{E} it follows that $M$ is $\mathbb{R}P^2$ or a Klein bottle.
We will show that in these two cases the final curve in the isotopy produced via the algorithm 
in Theorem $1.1^\prime$ has the same orientation as the final curve
in the homotopy.

Assume $M$ is $\mathbb{R}P^2$ or a Klein bottle and let $\alpha \in M$ be 
a non-contractible curve, such that $\alpha$ is homotopic through free loops to $-\alpha$.
We then have that $\alpha$ represents an element of order $2$ in $H_1(M, \mathbb{Z})$.
Let $p: \overline{M} \rightarrow M$ be a two-fold orientation covering
with $\overline{M}=S^2$ (if $M = \mathbb{R} P^2$) or $\overline{M}=T^2$ (if $M$ is the Klein bottle), and let 
$f: \overline{M} \rightarrow \overline{M}$ be the corresponding orientation reversing 
deck transformation. 

We then have that $p^{-1}(\alpha)$ is a collection of at most 2 simple closed curves and represents
a trivial element in $H_1(\overline{M}, \mathbb{Z})$. 
$p^{-1}(\alpha)$ has one connected component when $M = \mathbb{R}P^2$ and
two connected components when $M$ is a Klein bottle. 
In either case, $p^{-1}(\alpha)$ separates $M$
into two connected sub-manifolds with boundaries.  Let these sub-manifolds be $N_1$ and $N_2$;
we have that $\partial N_1 = \partial N_2 = p^{-1}(\alpha)$.

Let $\gamma$ be a generic homotopy of closed curves 
 on $M$ with $\gamma_0$ and $\gamma_1$ simple and let $\gamma'$ be the isotopy produced by the algorithm
in Theorem $1.1^\prime$. 
Without any loss of generality we may assume that for each $t$, $\gamma'_t$ is $\epsilon-$image equivalent
to $\gamma_t$. Indeed, we can modify the homotopy $\gamma$ by going back and forth in time over certain sections
of the homotopy so that this is true.

We lift $\gamma$ and $\gamma'$ to homotopies of curves on $\overline{M}$, obtaining 
$\overline{\gamma}$ and $\overline{\gamma'}$, respectively. For each time $t$
the collection of curves $\overline{\gamma}_t$ and $\overline{\gamma'}_t$ are fixed
by the deck transformation map $f$.
Fix an orientation of $\overline{M}$ and let $N$ be an embedded submanifold of $\overline{M}$
filling $\overline{\gamma}_0 = \overline{\gamma'}_0$. 
In particular, the orientation of $N$ is inherited from $\overline{M}$ and the orientation of $\partial N$
coincides with the orientation of $\overline{\gamma}_0$.
 
By the Isotopy Extension Theorem (see Theorem 1.3 in \cite{H} on p.180),
the isotopy of closed curves $\overline{\gamma'}$ can be extended to 
an isotopy $g': [0,1] \times N \rightarrow \overline{M}$ with 
$g'_0(N)= N$ and $g'_t(\partial N) = \overline{\gamma'}_t$.

We can also extend $\overline{\gamma}$ to a homotopy $g: [0,1] \times N \rightarrow \overline{M}$
with  $g_t(\partial N) = \overline{\gamma}_t$ as follows. 
Let $N_c$ be a manifold obtained from $N$ by gluing a collar $C= \partial N \times [0,1]$
to the boundary of $N$. We define a map $h_t: N_c \rightarrow \overline{M}$, such that
when restricted to $N \subset N_c$, $h_t$ is the inclusion map $\iota:N \rightarrow \overline{M}$
and for a point $(s,\tau) \in \partial N \times [0,1]$, we define 

	$$ h_t(s,\tau)=\overline{\gamma}_{t \tau}(s).$$
In other words, the map $h_t$ restricted to the collar $C$ coincides with 
$\overline{\gamma}$ restricted to $[0,t] \times S^1$
up to a rescaling of the domain. By composing this map with a diffeomorphism between $N$ and $N_c$,
we obtain the desired homotopy $g$.

We will also need a notion of the degree of a map at a point $y$.
Let $F:N \rightarrow \overline{M}$ be a smooth map and suppose $y \notin F(\partial N)$ 
is a regular point of $F$. Then the degree of $F$ at $y$ is given by
$$deg_F(y) = \sum _{x \in F^{-1}(y)} sign(dF_x)$$
where $sign(dF_x)=+1$ if $dF_x$ preserves orientation and is $-1$ otherwise.

Observe that $deg_F$ is constant on each connected component of
$M \setminus F(\partial N)$. Furthermore, suppose we choose a local chart on $\overline{M}$, containing a segment 
$l$ of $F(\partial N)$ with tangent vector pointing upwards (see Figure \ref*{fig:degree}) and the orientation
induced from that on $\overline{M}$ is counter-clockwise, then for two generic points $x$ and $y$ to the left and to the right
of $l$ respectively, we have $deg_F(x) = deg_F (y)+1$.

\begin{figure}[center] 
\includegraphics[scale=0.25]{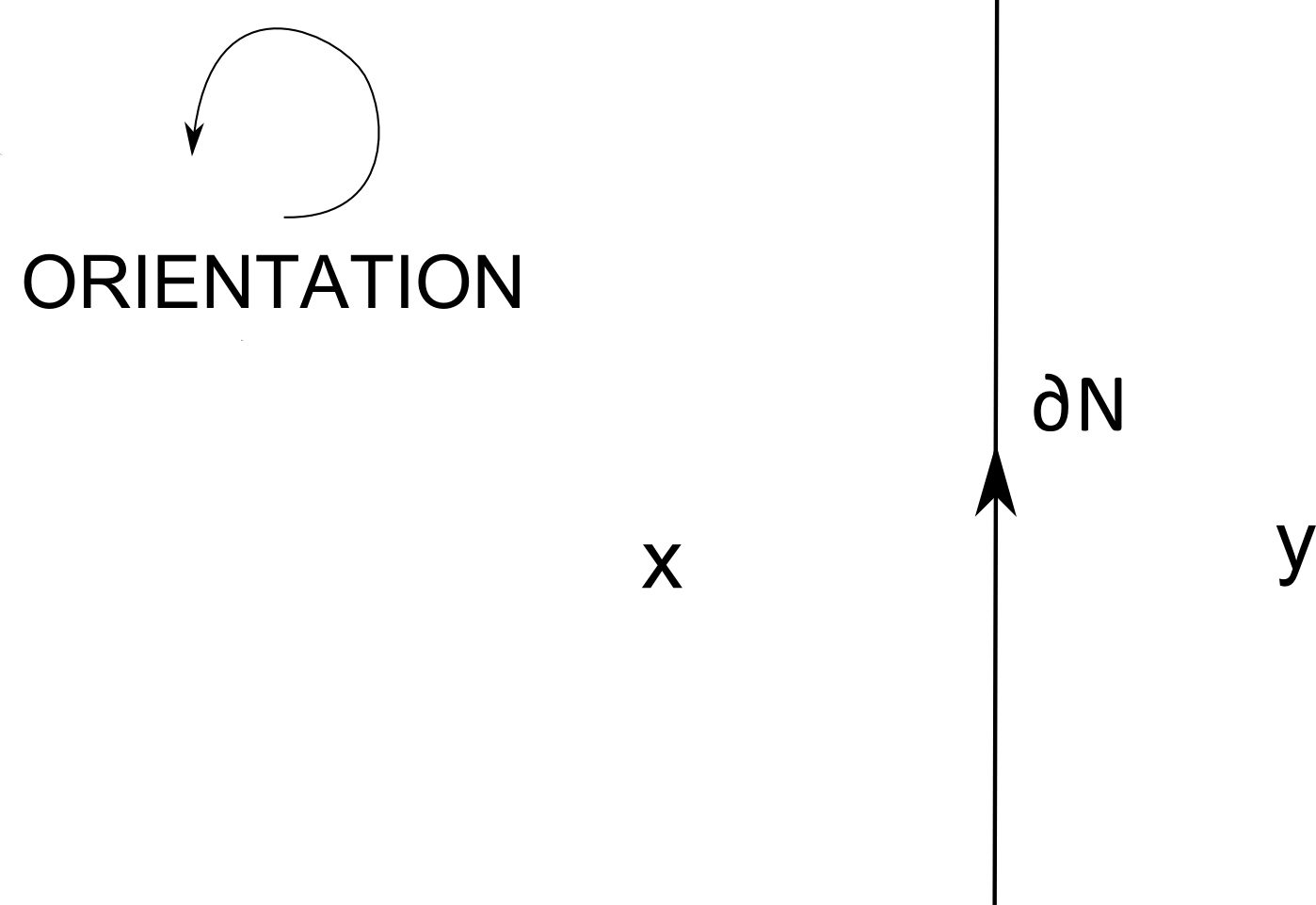} 
\caption{Crossing the boundary changes degree by 1: $deg_g(x) = deg_g (y)+1$} \label{fig:degree}
\end{figure}

We make two observations about the degrees $deg_{g_t}$ and $deg_{g'_t}$:

Observation 1. Define $deg_t=\sum  deg_{g_t}(x)$, where the sum is taken over all connected components
of $\overline{M} \setminus g_t(\partial N)$. We then have that $deg_t=1$ for all $t$.  This also holds for
$deg'_t = \sum deg_{g'_t} (x)$.

\textit{Proof} Since $deg_0 = deg'_0 = 1$, it is enough to consider how the degree changes under Reidemeister moves
(see Figure \ref*{fig:reidemeister_moves}).
The image of the boundary $\partial N$ is fixed by the action of orientation reversing
deck transformation $f$. So for each Reidemeister move $R$ happening
in the neighbourhood of a point $p \in \overline{M}$ there is a simultaneous
Reidemeister move happening in the neighbourhood of $f(p)$.
Considering each of the 3 types of Reidemeister moves we obtain that
changes to the sum of degrees $deg_t$ corresponding to each of the two simultaneous moves 
cancel out.

Observation 2. For each $x \in \overline{M}$, $deg_{g_t}(x)+deg_{g'_t}(x)=0 \; (mod \; 2)$.

\textit{Proof} Again this follows by induction on the number of Reidemeister moves.
We consider each of the R3 moves
and the corresponding modified moves from the proof of Theorem $1.1^\prime$
(see Figures \ref*{F-R} and \ref*{fig:M2b}).

Since, by assumption, $\gamma_1$ is a simple closed curve, we have that
$g_1(\partial N) = \pm g'_1(\partial N)$.
Since $\gamma_1$ is simple, it divides $\overline{M}$ into two connected regions, $M_1$ and $M_2$. 
The degree of $g_1$ is constant on each region, as is the degree of $g'_1$.  By Observation 1 and the fact that
the degree changes by 1 across $\partial M_1 = \partial M_2$, we have that the degree of $g_1$ is $0$ on one of
the regions, and is $1$ on the other.  By Observation 1 applied to $deg'_1$ and Observation 2, we see that
the degree of $g'_1$ on each region is the same as the degree of $g_1$. Since the degree of the map
on these two components determines the orientation 
of the boundary (see Figure \ref*{fig:degree}), we must have that
$g_t(\partial N) = g'_t(\partial N)$.
This finishes the proof.

\end{proof}

\textbf{Remark.} Any two simple curves on a 2-sphere are isotopic, however, 
Theorem \ref*{thm:baer_quantified} does not hold for simple curves on $S^2$.
Indeed, consider $S^2$ with the standard metric and orientation, and
consider a positively oriented circle $\gamma$ of length $\frac{1}{100}$. 
There exists a homotopy from $\gamma$ to $-\gamma$ through curves of length
$\leq \frac{1}{100}$, but any isotopy between them must go through a curve of length
$\geq 2 \pi$. We conjecture that, if a simple curve $\gamma$ is homotopic to a curve of opposite orientation through 
a homotopy of curves of length $<L$, then for any $\epsilon > 0$ there exists a contraction of 
$\gamma$ to a point through curves of length at most $L + \epsilon$.
Finally, we conjecture that Theorem \ref*{thm:baer_quantified} is 
true for contractible simple closed curves on a non-orientable surface.

\section{Homotopy of a double loop}

In this section we prove Theorem \ref*{thm:no_torsion_quantified}, the quantitative version of the
fact that orientable surfaces have no elements of order $2$ in their fundamental groups.

Let $M$ be an orientable surface and $\alpha$ a closed curve on $M$. 
Suppose $\gamma$ is a smooth homotopy with $\gamma_0 = 2 \alpha$.  Let $\alpha$ be parametrized by
$[0,1]$, with $\alpha(0) = \alpha(1)$.

Since $M$ is orientable, we can define a continuous unit normal $n(s)$ to $\alpha(s)$.
Let $\alpha_{\epsilon}(s)$ denote the value of the exponential map at $\alpha(s)$
evaluated at $ \epsilon n(s)$ for sufficiently small values of $\epsilon$.
Let $f$ be a smooth bump function supported (and non-zero) on
the open interval $(\frac{1}{2},1)$, and attaining maximum at $f(\frac{3}{4}) = 1$.
We define a homotopy of $2 \alpha$ to a curve $\beta$:

$$
\sigma_t(s) = \left \{
 \begin{array}{rl}
 \alpha(2s) &\mbox{ if $0 \leq s \leq \frac{1}{2}$} \\
 \alpha_{ t f(s) \epsilon}(2s-1)  &\mbox{ if $\frac{1}{2} \leq s \leq 1$}
       \end{array}
       \right.
$$

We can make a small perturbation to $\alpha$ that increases its length by an arbitrarily small amount, and
which removes all non-normal self-intersections (see the proof of Proposition \ref*{prop:generic}).  The original and perturbed versions of
$\alpha$ are clearly homotopic through curves that are larger than $L$ by an arbitrarily small amount.  As such,
choosing $\epsilon$ to be small enough, $\beta$ also has only normal self-intersections, and its length
is only larger than that of $\alpha$ by an arbitrarily small amount.
If $\alpha$ has $k$ self-intersections, then $\beta$ has $4k+1$ intersections, as in Figure \ref*{fig:double}.
Since $\beta$ is a small perturbation of $2 \alpha$, there exists a small perturbation of 
$\gamma$ that starts at $\beta$ and satisfies the conclusions of Lemma \ref*{lem:perturb}. 
For simplicity we will assume that $\gamma$ has already been put into this form.

\begin{figure}[center]
\includegraphics[scale=0.8]{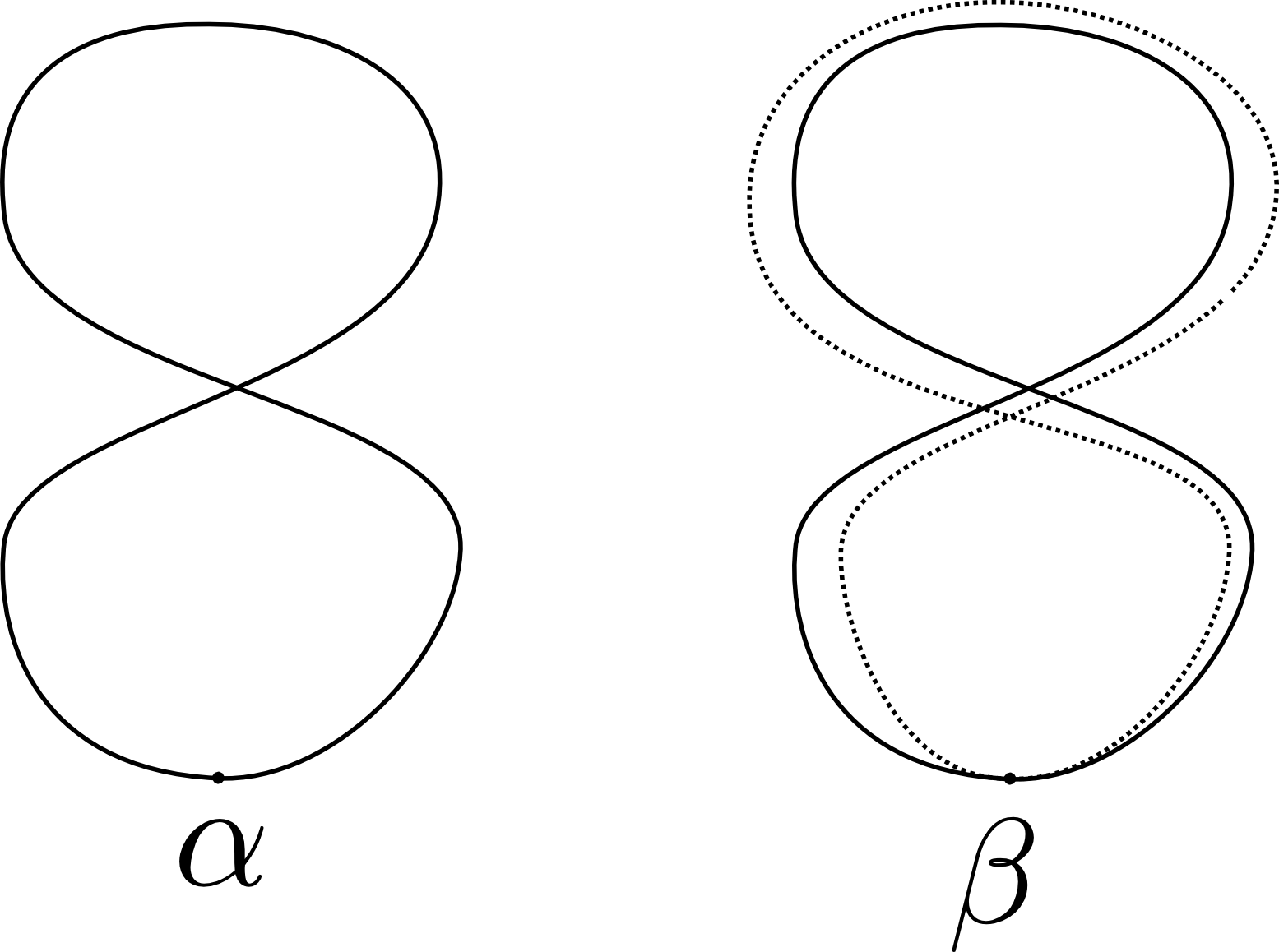} 
\caption{The curves $\beta$ and $\alpha$, if $\alpha$ has one intersection} \label{fig:double}
\end{figure}

To proof the theorem we need to construct a homotopy of $\beta([0, \frac{1}{2}])$ to
a point that goes through subcurves of $\gamma$.

The method that we use will be similar to that employed to prove Theorem $1.1^\prime$:
we will construct a graph $\Gamma$ which contains a path that will be used to define the desired homotopy. 
As in the proof of this theorem, let
	$$ 0 = t_0 < \dots < t_n = 1 $$
be times at which $\gamma_{t_i}$ contains only normal self-intersections, and between which exactly one Reidemeister move
occurs.  Let $\{p_i ^j\}$  be the set of intersection points of $\gamma_{t_j}$,
$1 \leq i \leq k_j$.

As before, we will begin by defining how vertices are added to $\Gamma$.  For each $p_i^j$,
add a vertex to $\Gamma$, and label it ``$p_i^j$".  Additionally, for each $j$ with $0 \leq j \leq n$,
add a vertex to $\Gamma$ with label ``$p_0 ^j$" (these will correspond to an ``empty self-intersection" at each
point $t_j$).

We will now define the edges of $\Gamma$ as follows. Let $R_j$ denote the Reidemeister move between $t_j$ and $t_{j+1}$.
Again, we will add certain ``vertical" edges between vertices with labels $p_i^j$ and $p_k^{j+1}$, followed
by certain ``horizontal" edges between vertices with labels $p_i^j$ and $p_k^j$.

\noindent \textbf{Vertical edges.}
Suppose first that the Reidemeister move $R_j$ does not involve
the intersection point $p^j_i$, $j < n$.  We then observe that the intersection point changes smoothly
through the Reidemeister move and can be uniquely matched with an intersection point 
$p^{j+1}_k$.  We connect the corresponding vertices of $\Gamma$ by an edge.
If $R_j = M1$ and $p^{j+1}_i$ was created in this move, then we connect vertex $p^j_0$
and $p^{j+1}_i$ by an edge. Similarly, if $R_j = -M1$ and $p^j_i$ is destroyed in this move, then
we connect $p^j_i$ to $p^{j+1}_0$. If $R_j = M3$, then we can add $3$ edges to the graph.
There are $3$ vertices $p^j_{i_1}$, $p^j_{i_2}$, and $p^j_{i_3}$ that are involved in $R_j$.
We can smoothly follow each of these intersections through the $R_j$ move.  Doing this,
we see that $p^j_{i_1}$ arrives at intersection point $p^{j+1}_{i_1'}$, 
$p^j_{i_2}$ arrives at point $p^{j+1}_{i_2'}$, and $p^j_{i_3}$ arrives at $p^{j+1}_{i_3'}$.  We add edges
connecting each of these pairs.  This is shown in Figure \ref*{fig:R3labels}.

\begin{figure}[center]
\includegraphics[scale=0.75]{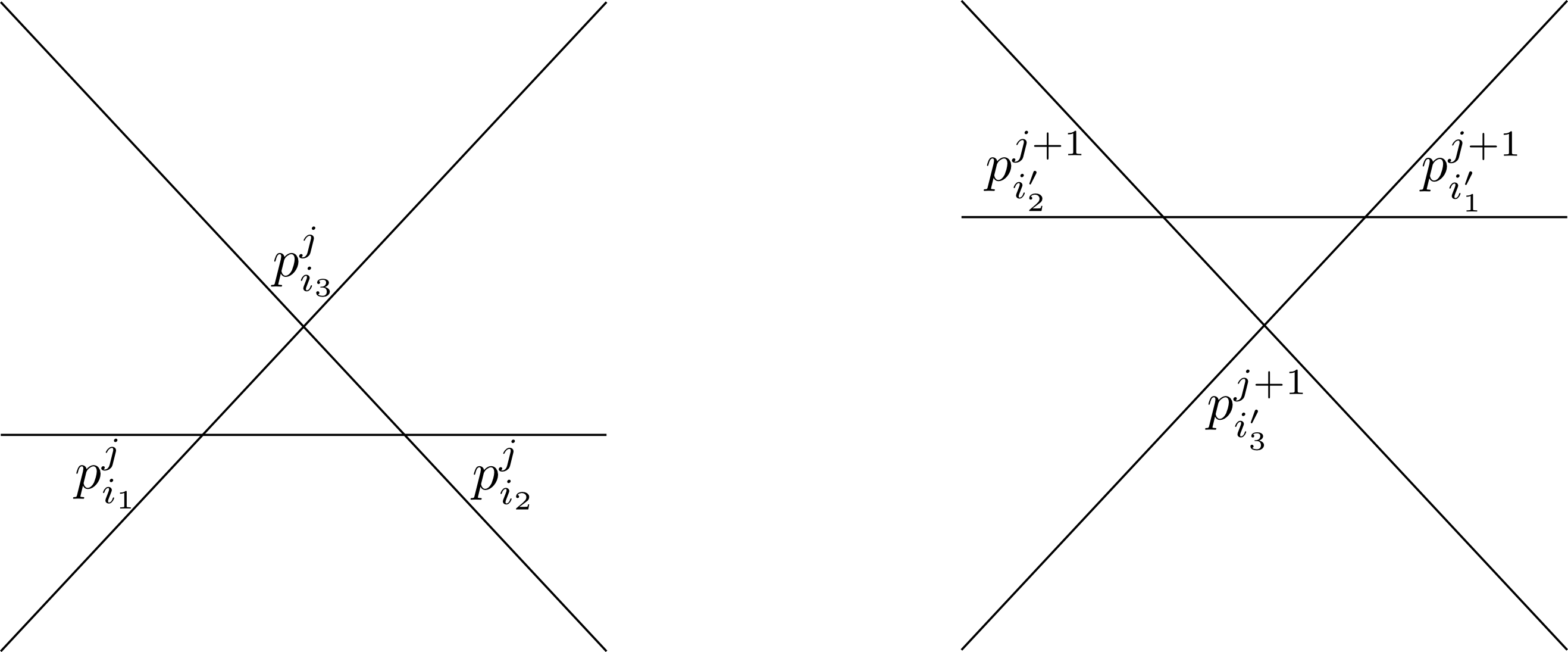} 
\caption{Labelling of vertices before and after an $R3$ move} \label{fig:R3labels}
\end{figure}

\noindent \textbf{Horizontal edges.}
If $R_j = M2$ and $p^{j+1}_{i_1}$ and $p^{j+1} _{i_2}$ are created in this move,
or if $R_{j+1} = -M2$ and $p^{j+1}_{i_1}$ and $p^{j+1}_{i_2}$ are destroyed in this move,
then we add an edge between these two vertices.
If both of these things happen, then we connect them by two edges.

\textbf{Edges of $\beta$.} We now define some additional edges between vertices $p_i^0$, $i \geq 1$,
which correspond to self-intersection points of $\beta$.
We will do this with the aid of Figure \ref*{fig:extra_edge}.  This figure shows $2$ arcs from
$\beta([0, \frac{1}{2}])$ in bold, and the corresponding $2$ arcs from $\beta([\frac{1}{2},1])$
as doted lines.  Let $p_{i_1}^0$, $p_{i_2}^0$, $p_{i_3}^0$, and $p_{i_4}^0$ be as in the figure.
$p_{i_1}^0$ and $p_{i_2}^0$ correspond to a double intersection of $2 \alpha$, moved apart when 
$2 \alpha$ was perturbed into $\beta$, and $p_{i_3}^0$ and $p_{i_4}^0$ are created during this perturbation.
For each such cluster of $4$ self-intersection points, we add an edge between $p_{i_1}^0$ and $p_{i_2}^0$,
and we add an edge between $p_{i_3}^0$ and $p_{i_4}^0$.

\begin{figure}[center]
\includegraphics[scale=1.5]{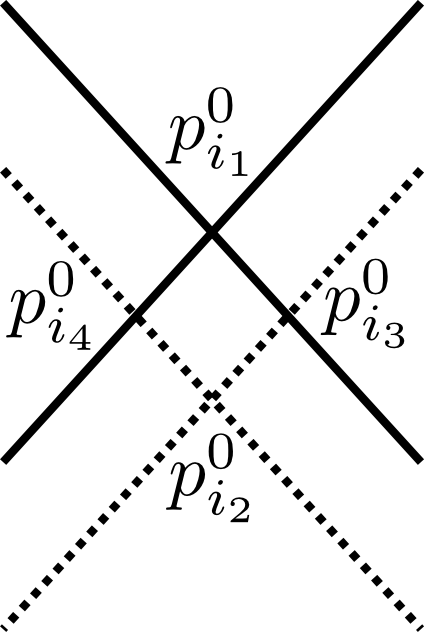} 
\caption{Adding extra edges to the graph $\Gamma$} \label{fig:extra_edge}
\end{figure}

Let $F$ denote the following subset of vertices of $\Gamma$
	$$ F = \bigcup_j p^j_0 \cup \bigcup_i p^n_i \cup p^0_1. $$
Observe that by construction, any vertex that is not in $F$
has exactly $2$ edges. Hence, there exists a path $P$ in $\Gamma$
from $p^0_1$ to a vertex in $F \setminus \{p^0_1 \}$.

We use $P$ to construct a homotopy of $\beta([0, \frac{1}{2}])$ to
a subcurve of $\gamma_1$ through subcurves of $\gamma$.
It may happen that in the process of this homotopy $\beta([0, \frac{1}{2}])$
is contracted to a point.

Note that $\gamma_{t_j}^{-1}(p^j _i)$ subdivides $S^1$ into two intervals $I_1$ and $I_2$.
Let $(C^j_i)_1$ and $(C^j_i)_2$ denote the two closed curves $\gamma_{t_j}|_{I_1}$ and $\gamma_{t_j}|_{I_2}$.
In the case of $p^j_0$, $(C^j_0)_{ \{ 1, 2 \} }$ are a point and the entire curve.
The next lemma will show us that paths in $\Gamma$ correspond to homotopies through subcurves.

\begin{lemma}
Suppose $p_i^j$ and $p_l^m$ are connected by an edge in $\Gamma$.
We can then find $a, b \in \{ 1, 2 \}$ with $a \neq b$, and $c, d \in \{ 1, 2 \}$ with
$c \neq d$, such that $(C^j_i)_a$ and $(C^m_l)_c$ are homotopic through subcurves of $\gamma$,
and $(C^j_i)_b$ and $(C^m_l)_d$ are also homotopic through subcurves of $\gamma$.
\end{lemma}
\begin{proof}
The lemma is evident for vertical edges. For horizontal edges, if $x$ and $y$ are the two intersection points
that are involved in the move $R = M2$ or $-M2$, we see that there is a smooth path from $x$ to the point of tangential
self-intersection at the point in time where $R$ occurs, and there is also a smooth path from $y$ to this point of tangential
self-intersection, and so we can find a smooth path from $x$ to $y$.  This path induces the desired homotopy.

We are left with the extra edges that were added in at time $t_0$. Looking again at Figure \ref*{fig:extra_edge},
we have $2$ pairs of self-intersection points.  We define a new perturbation of $2 \alpha$, called $\overline{\sigma}_t$:

$$
\overline{\sigma}_t(s) = \left \{
 \begin{array}{rl}
 \alpha_{ f(s+\frac{1}{2}) t \epsilon}(2s) &\mbox{ if $0 \leq s \leq \frac{1}{2}$} \\
 \alpha(2s -1 )  &\mbox{ if $\frac{1}{2} \leq s \leq 1$}
       \end{array}
       \right.
$$
Recall that $\alpha_{ f(s+\frac{1}{2}) t \epsilon}$ means a perturbation by ${ f(s+\frac{1}{2}) t \epsilon} n(s)$, where $n(s)$ is a smooth outward normal field and $f$ is a smooth bump function supported on $(\frac{1}{2},1)$.
We can now concatenate $\sigma_t$ with $\overline{\sigma_t}$.  The result is a homotopy from $\beta$ to itself, but with a different
parametrization.
Furthermore, for each cluster of self-intersection points as shown in Figure \ref*{fig:extra_edge}, we can follow each of them
over the course of the concatenated homotopy, and we can also follow the subcurves defined by each of them.  The result is exactly what we seek:
for each pair of self-intersections joined by an edge, they switch places over the course of this homotopy.  As such,
we obtain the desired homotopies over subcurves.
\end{proof}

By examining the degrees of the vertices we see that there exists a path from $p^0 _1$ to some other vertex in $F$.
Hence, there exists a homotopy of $\beta([0, \frac{1}{2}])$ that either contracts it to a point,
or homotopes it to $\gamma_t$ for some time $t$ over curves which are longer than $L$ by an arbitrarily small amount.
Since $\gamma_t$ is also contractible over curves with this same length bound,
this completes the proof of Theorem $1.2^\prime$.  From the remarks in the introduction, we have that Theorem \ref*{thm:no_torsion_quantified} is true, too.

\vspace{0.1in}
\textbf{Remark.} It is natural to ask whether the same result holds for $m$ iterates of a curve.
The above proof does not directly generalize to this situation. Consider, for example, a simple closed curve 
$\alpha$, such that $\beta = 3 \alpha$ can be contracted to a point through curves of length 
$\leq L$. A small perturbation of $\beta$ (analogous to $\sigma_t$
in the proof above) has two self-intersection points. Each of the self-intersection points
divides $\beta$ into two subcurves, one of which is a perturbation of $\alpha$
and the other is a perturbation of $2 \alpha$. Denote these subcurves by $A_1$ and $A_2$
for the first self-intersection point, and $B_1$ and $B_2$ for the second.

We can consider homotopies of these 4 subcurves
and apply the same method of tracking the self-intersection
point as in the proof above. For certain homotopies of $\beta$, however, this will
produce a homotopy between pairs $A_1$ and $B_1$, $A_2$ and $B_2$, 
giving us no information on how $\alpha$ can be contracted.

The authors have partial results on this problem that will appear in an upcoming article.

\section{Higher dimensional counter-examples}

In this section we give some higher dimensional counterexamples to Theorems \ref*{thm:no_torsion_quantified} and $1.2^\prime$.

We remark that Theorems \ref*{thm:baer_quantified} and $1.1^\prime$  
do not hold for curves on a manifold of dimension 3 because of the existence
of non-trivial knots. In higher dimensions they trivially hold since curves in a generic homotopy
do not have self-intersections.

We begin with a definition:
\begin{definition}
Suppose $\alpha$ is a curve of length $\leq L$, and let
	$$ [\alpha]_L = \inf \{ \textrm{length}(\beta) : \text{ $\beta$ is homotopic to $\alpha$
through curves of length at most $L$}\}. $$
\end{definition}

Theorem \ref*{thm:no_torsion_quantified} implies the following:
\begin{proposition}
If $M$ is an orientable 2-manifold then, for every $\epsilon > 0$,
$[\alpha] _{L+\epsilon} \leq [2 \alpha] _L.$
\end{proposition}

This is not true in higher dimensions.

\begin{proposition}
\label{prop:n}
For every $K > L > 0$ there exists a Riemannian metric on $S^n$, $n \geq 4$,
and a curve $\alpha$ on $S^n$ such that $[2 \alpha]_L = 0$ and $[\alpha]_K > 0$.
\end{proposition}


\begin{proof}
Consider an embedding $F$ of $\R P^2$ in $\R^n$, $n \geq 4$. 
By the tubular neighbourhood theorem (see, for example, \cite{H})
for a small number $l$ the set $T = \{x : \; dist(x, \R P^2) < l\}$
is diffeomorphic to the normal bundle of the embedding
with fibers $D^{n-2}$, the $n-2$ dimensional disc.

Let $U$ be a small open disc in $\R P^2$ and consider a local coordinate system 
$V=U \times D^{n-2}$.  The chart that we are using takes a point $(x,r,\theta) \in V$, where $x \in U$, $r \in [0,1]$,
$\theta \in S^{n - 3}$, and maps it to
	$$ F(x) + l r \theta, $$
where $lr \theta$ lies in the $n-2$ dimensional normal space to $F(\mathbb{R}P^2)$ at $F(x)$. Let $ds^2 = f_1 dx^2 + f_2 dr^2 +f_3 d \theta ^2$
be the metric induced by the embedding, and let $g(t)$ be a rapidly growing increasing function
 that is equal to $1$ at $0$. Define a new metric $ds_1 ^2=
 g(r) f_1 dx^2 + g(r) f_2 dr^2 +f_3 d \theta ^2$. This extends to a well-defined metric
 on all of $T$ with the property that projection $\pi$ on $\R P^2$ is $1-$Lipschitz
 and the distance from $\R P^2$ to $\partial T$ is at least
	$$ \int_0 ^l g(r) dr >> K. $$

Let $\alpha$ be a systole (shortest non-contractible curve) on $\R P^2$.
Let $L$ be the maximal length of a curve in a contraction of $2 \alpha$ to a point.
Since $T$ retracts onto $\R P^2$, $\alpha$ is not contractible in $T$.
Hence, any homotopy contracting $\alpha$ to a point must go through a curve intersecting
$\partial T$. Let $\beta$ be the first such curve in the homotopy, and
let $S_{\frac{l}{2}} = \{x \in T : \; dist (x, \R P^2) =\frac{l}{2} \}$, where $dist$
is the Euclidean distance. If $\beta$ intersects $S_{\frac{l}{2}}$ then 
$$ length(\beta) \geq \int_\frac{l}{2} ^l g(r) dr > K,$$ 
otherwise $\beta$ lies entirely in the set $\{x \in T : \; dist (x, \R P^2) >\frac{l}{2} \}$
and so
$$length(\beta) \geq  g(\frac{l}{2})length(\alpha) > K.$$

On the other hand, $2 \alpha$ contracts to a point on $\R P^2$
through curves of length $\leq L$.  By rescaling this example, we can choose any value we like for $L$ and $K$.

Since everything happens in some bounded subset of $\mathbb{R}^n$,
the statement of the theorem is true for any smooth manifold.
\end{proof}

In dimension $3$ we prove a somewhat weaker statement.  


\begin{proposition}
\label{prop:3}
For every $K>L>0$, there exists a Riemannian metric on $S^3$
and a simple closed curve $\alpha \subset S^3$ such that $[\alpha]_K > [2 \alpha]_L$.
\end{proposition}

\begin{proof}
Let $Q$ be a constant much larger than $K$, and let $T$ be a 3-dimensional manifold with boundary 
obtained from the parallelepiped $[-1,1] \times [-Q,Q]^2$ by identifying the
sides 
$-1 \times [-Q,Q]^2$ and $1 \times [-Q,Q]^2$ by a map 
$p: (1, v,w) \rightarrow (-1, -v,-w)$.
Let $f$ be an embedding of $T$ in $\R^3$ and modify the metric on $\R^3$
so that the embedding is an isometry.
Note that if we restrict $f$ to the  central slice $M=[-1,1] \times (-Q,Q) \times 0$
module $p$, we obtain an embedding of the M\"{o}bius strip.
The interior of $T$ is then diffeomorphic
to the normal bundle for this embedding.
We can think of $T$ as a fiber bundle over $M$ with fibers $[-Q,Q]$.
 
Let $\alpha$ be the central circle of $M$. Consider a homotopy contracting $\alpha$ to
any curve of length $<1$. 
Since $T$ can be retracted to $M$, there is no such curve in the homotopy class of $\alpha$ in $T$,
and so there will be a curve $\beta$ in the homotopy
that intersects $\partial T$ for the first time.  We make two observations.  First, since
$T \setminus M$ can be retracted to a curve homotopic to $2 \alpha$ inside $T$,
and since $\beta$ is not homotopic to any multiple of $2 \alpha$ inside $T$, $\beta$ must intersect $M$.
Second, if $M' = [-1,1] \times 0 \times [-Q, Q]$, then $T \setminus M'$ also has this property,
and so $\beta$ intersects $M'$ as well.  If $\beta$ intersects $f(M \times \partial [-Q, Q])$, then since $\beta$
goes through $M$, $\beta$ must be at least $Q$ in length.  Similarly, if $\beta$ intersects $f(\partial M \times [-Q,Q])$,
then since it goes through $M'$, it must also be at least $Q$ in length.

Consider $2 \alpha$. It can escape $T$ without exceeding length $2$.
Let $\gamma$ be a curve linked with $M$, but not intersecting $M$, as in Figure \ref*{mobius}. 
By modifying the metric in a small neighbourhood of $\gamma$, we can make 
$length(\gamma) < \epsilon$ for any $\epsilon > 0$, and we can homotope $2 \alpha$ to $\gamma$
through curves of length $< 100$. By scaling this example appropriately we obtain 
 the result for any values of $K$ and $L$.  Since we have only locally modified the metric, this holds for
any three dimensional manifold.

\begin{figure}[center] 
\includegraphics[scale=0.4]{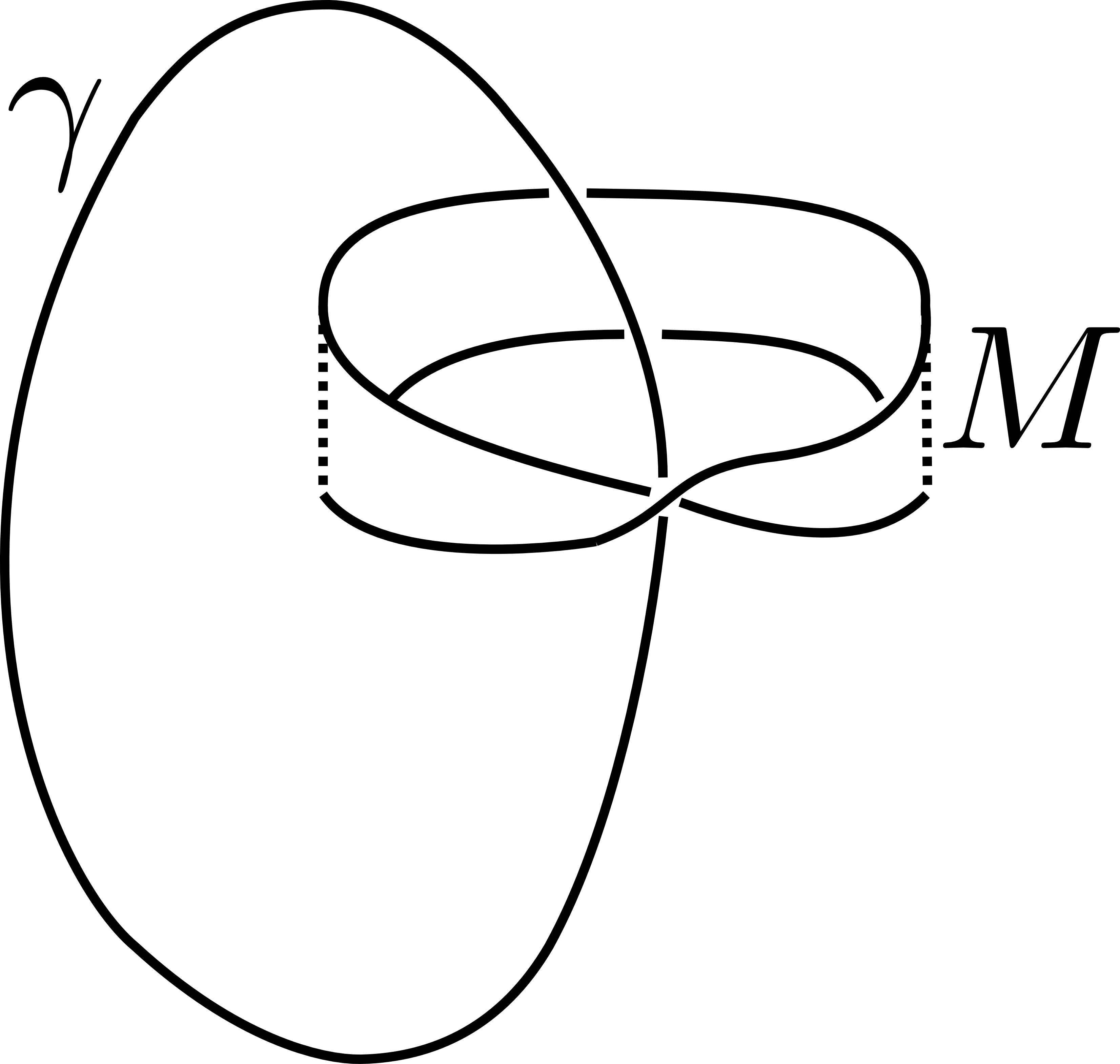} 
\caption{A curve $\gamma$ linked with $M$} \label{mobius}
\end{figure}
\end{proof}


\begin{tabbing}
\hspace*{7.5cm}\=\kill
Gregory R. Chambers                    \> Yevgeny Liokumovich\\
Department of Mathematics           \> Department of Mathematics\\
University of Toronto               \> University of Toronto\\
Toronto, Ontario M5S 2E4            \> Toronto, Ontario M5S 2E4\\
Canada                              \> Canada\\
e-mail: chambers@math.utoronto.ca   \> e-mail: liokumovich@math.utoronto.ca\\
\end{tabbing}

\end{document}